\documentclass[12pt,letterpaper]{amsart}
 \usepackage[leqno]{amsmath}
 \usepackage{amsfonts}
 \usepackage{amssymb}
 \usepackage{amsthm}
 \usepackage{amssymb}
 \usepackage[all]{xy}
 \usepackage{graphicx}
 \usepackage{xcolor}
 \usepackage{hyperref}
 \usepackage{lipsum}
 \usepackage{lineno}
 \usepackage[normalem]{ulem}
 \usepackage{tikz-cd} 
 \usepackage{multicol}
 \usepackage[numbers]{natbib}
 
 \usepackage{scalerel,stackengine}
 \stackMath
 \newcommand\reallywidehat[1]{%
 	\savestack{\tmpbox}{\stretchto{%
 			\scaleto{%
 				\scalerel*[\widthof{\ensuremath{#1}}]{\kern-.6pt\bigwedge\kern-.6pt}%
 				{\rule[-\textheight/2]{1ex}{\textheight}}
 			}{\textheight}%
 		}{0.5ex}}%
 	\stackon[1pt]{#1}{\tmpbox}%
 }
 \newcommand{\la}{\langle}
 \newcommand{\ra}{\rangle}

 \newcommand{\slnr}{\mathfrak{sl}_n(R)}

 \newcommand{\sltwo}{\mathfrak{sl}_2}

 \theoremstyle{plain}
 \newtheorem{theorem}{Theorem}[section]
 \newtheorem*{theorem*}{Theorem}
 \newtheorem*{claim*}{\textit{Claim}}

 \newtheorem{prop}[theorem]{Proposition}
 \newtheorem{lemma}[theorem]{Lemma}

 \theoremstyle{definition}
 \newtheorem{definition}[theorem]{Definition}
 
 \newtheorem{example}[theorem]{Example}
 
 \newtheorem{remark}[theorem]{Remark}

 \numberwithin{equation}{section}

\begin{document}

\title{Finite presentability of  universal central extensions of ${\mathfrak{sl}_n}$,  II }
\author{Zezhou Zhang}

\begin{abstract}
	In this note we connect finite presentability of a Jordan algebra to finite presentability of its Tits-Kantor-Koecher algebra. Through this we complete our discussion of finite presentability of universal central extensions of ${\mathfrak{sl}_n(A)}$, $A$ a $k$-algebra, initiated in \cite{ZeZel}, and answer a question raised by Shestakov-Zelmanov \cite{ShestZelFPjordan} in the positive.  
\end{abstract}
\renewcommand{\thesubsection}{\Alph{subsection}}

\maketitle


Throughout this note all algebras are considered over a field $k$ containing $\frac{1}{2}$.

\section{Introduction}

Let $\mathcal{V}$ be a variety (of universal algebras) in the sense of \cite{JacobsonJordanbook}, \cite{RingsNearlyAssociative}. An algebra $A\in \mathfrak{V}$ is said to be \textit{finitely presented (f.p.)} if it can be presented in $\mathcal{V}$ by finitely many generators and finitely many relations.

\begin{definition} 	
	A $k$-algebra $J$ satisfying the identities \begin{enumerate}
		\item $xy=yx$,
		\item $(x^2y)x=x^2(yx)$
	\end{enumerate}for all $x, y \in J$ is called a \textit{Jordan} algebra.  
\end{definition}

\begin{remark}
	$J$ as above is sometimes referred to as a linear Jordan algebra, in contrast to the concept of quadratic Jordan algebras. These two concepts are equivalent when $\tfrac{1}{2} \in k$.
\end{remark}

\begin{example}
	An associative $k$-algebra $A$ admits a  canonical Jordan product given by $x\circ y= \frac{1}{2}(xy+yx)$. This new product on $A$ makes it a Jordan algebra, denoted by $A^{(+)}$.
\end{example}

J.M. Osborn (see \cite{HersteinRingswithinvolution}) showed that for a finitely generated associative algebra $A$ the Jordan algebra $A^{(+)}$ is finitely generated.

In \cite{ShestZelFPjordan},  Shestakov and Zelmanov considered the question whether for a finitely presented associative algebra $A$ the Jordan algebra $A^{(+)}$ is finitely presented. They proved (among other things) that 
\begin{enumerate}
	\item the Jordan algebra $k\la x, y \ra ^{(+)}$, where $k\la x, y \ra$ is the free associative algebra of rank $2$, is not finitely presented;
	\item let $A$ be a finitely presented associative algebra and let 
	$M_n(A)$ be the algebra of $n\times n$ matrices over $A$, and $n \geq 3$, then the Jordan algebra  $M_n(A)^{(+)}$ is finitely presented.
\end{enumerate}
For the borderline case of $M_2(A)$ they asked if the Jordan algebra $M_2(A)^{(+)}$ is finitely presented.

In this paper we give a positive answer to this question.

\begin{theorem}
Let $A$ be a finitely presented associative $k$-algebra. Then the Jordan algebra $ M_2(A)^{(+)}$ is finitely presented.
\end{theorem}

The proof of this theorem is based on the paper \cite{ZeZel} on universal central extensions of Lie algebras $\mathfrak{sl}_n(A)$.

\begin{definition}
Let $A$ be a finitely presented associative algebra. $\mathfrak{sl}_n(A)$ is the Lie algebra generated by off-diagonal matrix units among $n \times n$ matrices; forming a subalgebra of $\mathfrak{gl}_n(A)$. Equivalently,  $\slnr=\{X\in \mathfrak{gl}_n(A) \mid \mbox{tr} (X) \in [A,A] \}$. 
\end{definition}

\begin{definition}
Let $\mathcal{L}, \mathfrak{g}$ be $k$-Lie algebras where $\mathfrak{g}$ is perfect.  A surjective Lie homomorphism $\pi: \mathcal{L} \rightarrow \mathfrak{g} $ is called a central extension of $\mathfrak{g}$ if $\ker(\pi)$ is central in $\mathcal{L}$. This $\pi: \mathcal{L} \rightarrow \mathfrak{g} $ is called a universal central extension if there exists a unique homomorphism $\phi: \mathcal{L} \rightarrow \mathcal{M}$ from $f: \mathcal{L}\rightarrow \mathfrak{g}$ to any other central extension $p: \mathcal{M}\rightarrow \mathfrak{g}$ of $\mathfrak{g}$. In other words, $f=p\circ\phi$. The universal central extension of $\mathfrak{g}$ is customarily denoted as $\widehat{\mathfrak{g}}$. Perfectness of $\mathfrak{g}$ guarantees the existence of $\widehat{\mathfrak{g}}$, which is necessarily perfect (see \cite{Neher}).
\end{definition}

In \cite{ZeZel} it was shown that for a finitely presented associative algebra $A$ and $n\geq 3$, the Lie algebra $\widehat{\mathfrak{sl}_n(A)}$ is finitely presented.  As a consequence of the Lie-Jordan correspondence results of this paper, we may sharpen this result:

\begin{theorem}
	The universal central extension of ${\mathfrak{sl}_2(k\la x, y \ra)}$ is not finitely presented as a Lie algebra.
\end{theorem}

\begin{remark}
	The result in \cite{ZeZel} holds without the restriction  $\text{char}(F) \neq 2$.
\end{remark}

\section{Finite Presentation of Jordan systems}
\noindent\textit{\underline{Jordan triple systems}} 

A Jordan algebra $J$ is equipped with a Jordan triple product $\{a,b,c\}=(ab)c+a(bc)-b(ac)$. If $J=A^{(+)}$, where $A$ is an associative algebra, the $\{a,b,c\}=\tfrac{1}{2}(abc+cba)$.

\begin{definition}[See \cite{MeybergLectures}]
A vector space $V$ over a field $k$ containg $\tfrac{1}{2}$ is called a \textit{Jordan triple system} if it admits a trilinear product $\{ \ , \ , \}: V^3 \rightarrow V$ that is symmetric in the outer variables, while satisfying the identities 

\begin{enumerate}
	\item $\{a,b,\{a,c,a\}\}=\{a,\{b,a,c\},a\}$,
		\item $\{\{a,b,a\},b,c\}=\{a,\{b,a,b\},c\}$,
			\item $\{a,\{b,\{a,c,a\},b\},a\}=\{ \{a,b,a\}, c, \{a,b,a\} \}$
\end{enumerate}
for any $a,b,c \in V$. 
\end{definition}

\begin{remark}
When $\tfrac{1}{6} \in k$, the three defining identities of a Jordan triple system may be merged in to $\{a,b,\{c,d,e\}\}=\{  \{a,b,c\} ,d,e\}-\{c,\{b,a,d\},e\}+\{a,b,\{c,d,e\}\}$.
\end{remark}

It is easy to see that every Jordan algebra is a Jordan triple system with respect to the Jordan triple product $\{ \ ,\  ,\ \}$.

\noindent\textit{\underline{TKK constuction for a Jordan triple system.}}

 For a Jordan algebra $J$ we let $K(J)$ be the Tits-Kantor-Koecher (abrreviated TKK) Lie algebra of $J$ viewed as a Jordan triple system.
 
 \begin{definition}[See \cite{ShestZelFPjordan}. See also {\cite[Section 5]{BenkartSmirnovBC1}}, \cite{AllisonGaoUnitary}] \label{tkkdef}  
 Let $T$ be a Jordan triple system.  Let $\{e_u, u\in I\}$ be a basis of the vector space $T$. Let \[\{e_u,e_v, e_w\}=\sum\gamma^t_{uvw}e_u, \] where $\  t,u,v,w \in I$; $\gamma_{uvw}^t\in k$. 
 The Lie algebra $K(T)$ is presented by generators $x_u^{\pm}, u \in I $ and relations  
 \begin{gather}
 [x_u^\sigma, x_v^\sigma]=0,  \tag{K1}\label{K1}\\
 \quad \quad  \left[[x_u^\sigma, x_v^{-\sigma}], x_w^\sigma\right]-\sum\gamma^t_{uvw}e_t=0, \tag{K2}\label{K2}
 \end{gather}where $\  t,u,v,w \in I$; $\sigma=\pm$. This Lie algebra is called the \textit{(Universal) Tits-Kantor-Koecher Lie algebra associated to $T$}. It is obvious that this construction is basis-independent.
 \end{definition}

\begin{remark}
	The above (universal) TKK construction $T \rightarrow K(T)$ is a functor: see for example \cite{CavenySmirnovJordanLieCategories}. In other words, any homomorphism of Jordan Triple systems $T_1 \stackrel{\phi}{\rightarrow} T_2$ gives rise to a homomorphism 
	$K(T_1) \stackrel{\phi}{\rightarrow} K(T_2)$ of Lie algebras. A quick proof follows from the definition of $K(T)$ above: choose a basis $B_1 \cup B_2$ of $T_1$ such that $\phi(B_1)$ is linearly independent while $\phi(B_2)=0$. Extend $\phi(B_1)$ to a basis of $T_2$ and it is clear that we may extend $\phi$.
\end{remark}


\begin{lemma}\label{lemma1}
	Let $T$ be a Jordan triple system. If the Lie algebra $K(T)$ is finitely presented then the Jordan triple system is finitely presented as well.
\end{lemma}

\begin{proof}
	Let $\{e_u, u\in I\}$ be a basis of the vector space $T$ and let $ \{e_u,e_v, e_w\}=\sum\gamma^t_{uvw}e_u$, where $\  t,u,v,w \in I$; $\gamma_{uvw}^t\in k$. Define $K(T)$ as in Definition \ref{tkkdef}. For a subset $S\subset I$, let $R(S)$ be the set of those relations from (\ref{K1}),(\ref{K2}) that have all indices $t,u,v,w$ lying in $S$. 
	
	By our assumption there exists a finite subset $S\subset I$ such that the algebra $K(T)$ is generated by $x_u^\pm, \ u\in S$, and presented by the set of relations $R(S)$. Let $\widetilde{T}$ be the Jordan triple system presented by generators $y_u, \ u \in S$ and the set of relations \[R_J(S): \{y_u, y_v, y_w\}-\sum\gamma^t_{uvw}y_t=0,\] where $u,v,w, t \in S$. We claim that the mapping $y_u \mapsto e_u, \ u \in S$, extends to an isomorphism $\widetilde{T} \cong T$. Since the elements $e_u, u \in S$ satisfies the relations from $R_J(S)$ it follows that the mapping $y_u \stackrel{\varphi}{\rightarrow} e_u, \ u \in S$, extends to a homomorphism $\widetilde{T} \stackrel{\varphi}{\rightarrow}T$. This homomorphism gives rise to a homomorphism $K(\widetilde{T}) \stackrel{\varphi}{\rightarrow} K(T)$.
	
	Consider the Lie algebra $K(\widetilde{T})$. Since the elements $y_u^\pm \in K(\widetilde{T}), \ u \in S$ satisfy the relations $R(S)$,  the mapping $x_u^\pm \rightarrow y_u^\pm, u \in S$ extends to a homomorphism $K(T) \rightarrow K(\widetilde{T})$.  Hence the homomorphism $K(\widetilde{T}) \stackrel{\varphi}{\rightarrow}K(T)$ is an isomorphism. This implies that the homomorphism $\widetilde{T} \stackrel{\varphi}{\rightarrow} T$ is bijective, hence an isomorphism. 
\end{proof}

\begin{lemma}\label{lemma2}
	Let $J$ be a unital Jordan algebra. If $J$ is finitely presented as a Jordan triple system, then $J$ is finitely presented as a Jordan algebra.
\end{lemma}

\begin{proof}
	Let elements $a_1, \ldots, a_n \in J$ generate $J$ as a Jordan triple system. Then they clearly generate $J$ as a Jordan algebra, according to the formula $\{a,b,c\}=(ab)c+a(bc)-b(ac)$. 
	
	Let $\mathfrak{T}$ and $\mathfrak{J}$ be the free Jordan triple system and the free Jordan algebra on the set of free generators $x_u, \ u \geq 1$, respectively. Since $\mathfrak{J}$ is a Jordan triple system with respect to the Jordan triple product there exists a natural homomorphism $\mathfrak{T} \rightarrow \mathfrak{J} , \ a \mapsto \widetilde{a}$, that extends the identical mapping  $x_u \rightarrow x_u, \ i \geq 1$.
	
	Let $R \subset \mathfrak{T}$ be a finite subset, all elements from $R$ become zero when evaluated at $a_u, \ 1 \leq u \leq n$, and $R$ defines $J$ as a Jordan triple system. Since elements $a_1, \ldots a_n$ generate $J$ as a Jordan triple system, there exists an element $\omega(x_1, \ldots, x_n) \in \mathfrak{T}$ such that $\omega(a_1, \ldots, a_n)=1$.
	
	Let $P=\widetilde{R} \cup \{\omega(x_1, \ldots, x_n)^2- \omega(x_1, \ldots, x_n) \} \subset \mathfrak{J}$. Consider the Jordan algebra $\widetilde{J}=\la x_1, \ldots, x_n \mid P=(0) \ra$. Our aim is to show that $\widetilde{J} \cong J$.
	
	Since the generators $a_1, \ldots, a_n$ satisfy the relations $P$ it follows that there exists a surjective homomorphism $\widetilde{J} \stackrel{\varphi}{\rightarrow}J$,  $\varphi(x_i)=x_i, \ 1\leq i \leq n$. We claim that the elements $x_1, \ldots , x_n$ generate the Jordan algebra $\widetilde{J}$ as a Jordan triple system: Indeed, let $\widetilde{J}' $ be the Jordan triple system generated by $x_1 \ldots, x_n$ in $\widetilde{J}$. The relations $P$ imply that the element $\omega(x_1, \ldots, x_n)$ is an identity element of the algebra $\widetilde{J}$ and $\omega(x_1, \ldots, x_n) \in \widetilde{J}'$. If $a, b \in \widetilde{J}'$ then $ab=\{a,\omega(x_1, \ldots, x_n), b\} \in \widetilde{J}'$, which implies $\widetilde{J}' \cong \widetilde{J}$.
	
	Since the generator $x_1, \ldots, x_n$ of the Jordan triple system $\widetilde{J}$ satisfy the relations $R$ it follows that there exists a homomorphism of Jordan triple systems $J \stackrel{\psi}{\rightarrow} \widetilde{J}$ that extends $\psi(x_u)=x_u, \ 1 \leq u \leq n $. This implies that $\varphi, \psi$ are isomorphisms of Jordan triple systems. In particular, $\varphi$ is a bijection. Hence $\varphi$ is an isomorphism of Jordan algebras.\end{proof}

Lemmas \ref{lemma1} , \ref{lemma2} imply the following proposition:

\begin{prop}\label{Tkk Jor presentation}
	Let $J$ be a Jordan algebra with $1$. Then, if its TKK Lie algebra $K(J)$ is finitely presented, then $J$ is finitely presented.
\end{prop}


\section{Non-finite presentation of $\widehat{\mathfrak{sl}_2(A)}$}

\begin{theorem}\label{sl2 non fp theorem} 
Let $k$ be a field where $char(k)\neq 2$. Then $\widehat{\mathfrak{sl}_2(k\la x, y \ra)}$ is not finitely presented.
\end{theorem}

The universal central extension of $sl_2(A)$ was discussed in Kassel-Loday \cite{Kassellodaycentral}; see also \cite{GaoSt2}. To elaborate:

\begin{theorem}[See \cite{GaoSt2}]\label{st2pres}
	$\widehat{\sltwo(A)}$   admits a presentation where the Lie algebra is generated by  $\{X_{12}(a), X_{21}(a), T(a,b) \mid  a,b \in A \}$, subjecting to the relations
	\begin{align*}
	X_{ij}(\alpha a + \beta b)&= \alpha X_{ij}(a) + \beta X_{ij}(b),\\
	T(a,b)&=[X_{12}(a),X_{21}(b)],\\
	[T(a,b),X_{12}(c)]&=X_{12}(abc+cba),\\
	[T(a,b),X_{21}(c)]&=-X_{12}(bac+cab),\\
	[X_{ij}(A),X_{ij}(A)]&=0.
	\end{align*}
	for all $1\leq i\neq j \leq 2$, $a,b,c \in A$, $\alpha, \beta \in k$.
\end{theorem}
\begin{proof}[Proof of Theorem \ref{sl2 non fp theorem}]
	
	 Define $x_+(a):=X_{12}(a)$, and $x_-(a):=X_{21}(\tfrac{1}{2}a)$. These new expressions give a new presentation of $\widehat{\sltwo(A)}$, in generators  $\{x_+(a),x_-(a)|a\in A\}$:\begin{align*}
	a \mapsto x_+(a) \mbox{ and } a \mapsto x_-(a) \mbox{ are } k \mbox{ linear maps,}\\
	[x_+(A),x_+(A)]=[x_-(A),x_-(A)]=0,\\
	[[x_+(a),x_-(b)],x_+(c)]=x_+(\tfrac{1}{2}(abc+cba)),\\
	[[x_-(b),x_+(a)],x_-(c)]=x_-(\tfrac{1}{2} (bac+cab)),
	\end{align*}for all $a,b,c \in A$.  This gives $\widehat{\sltwo(A)} \cong K(A^{(+)})$, as this presentation defines $K(A^{(+)})$.
	
	Now set $A=k\la x, y \ra$. According to \cite{ShestZelFPjordan}, $k\la x, y\ra ^{(+)}$ is not finitely presented as a Jordan algebra. Theorem \ref{Tkk Jor presentation} then implies that $\widehat{\mathfrak{sl}_2(k\la x, y \ra)}$, being isomorphic to $K(k\la x, y \ra^{(+)})$, is not finitely presented.\end{proof}

\section{Finite presentation of $M_2(A)^{(+)}$}

\begin{lemma}
	Let $k$ be a field where $char(k)\neq 2$, $A$ a unital associative $k$-algebra. Then $\widehat{\mathfrak{sl}_4(A)} \cong K(M_2(A)^{(+)})$.
\end{lemma}

\begin{proof}
	According to \cite{Kassellodaycentral,GaoShang}, $\widehat{\mathfrak{sl}_4(A)}$ admits the presentation with generating set $\{X_{ij}(s) \mid s\in A, 1\leq i \neq j \leq n\}$ and set of  relations 
	\begin{align*} &\alpha\mapsto X_{ij}(\alpha) \text{ is a $k$-linear map,}\\
	&[X_{ij}(\alpha), X_{jk}(\beta)] = X_{ik}(\alpha\beta), \text{ for distinct } i, j, k, \\
	&[X_{ij}(\alpha), X_{kl}(\beta)] = 0, \text{ for } j\neq k, i\neq l,
	\end{align*}for all $\alpha, \beta \in A$.  According to \cite{ZeZel}, $\widehat{\mathfrak{sl}_4(A)}$ is finitely presented.
	
	Like in Theorem \ref{sl2 non fp theorem}, we reorganize the presentation: the $2 \times 2$ block partition of $\mathfrak{sl}_4$ allows us to identify $X_{13}(A) \oplus X_{14}(A) \oplus X_{23}(A) \oplus X_{24}(A)$ to a copy of $M_2(A)^{(+)}$ (`` $x_+$''), and $X_{31}(A) \oplus X_{32}(A) \oplus X_{41}(A) \oplus X_{42}(A)$ to another copy of $M_2(A)^{(+)}$ (`` $x_-$''). Now identify through the ($k$-linear) assignment 
	\[\begin{tabular}{c c c c}
	$X_{13}(a)  \rightarrow x_+(e_{11}(a)),$ & $X_{14}(a)  \rightarrow x_+(e_{12}(a)),$  & $X_{23}(a)  \rightarrow x_+(e_{12}(a)),$ & $X_{24}(a)  \rightarrow x_+(e_{22}(a))$
	\end{tabular}\]  and 
	\[\begin{tabular}{c c c c}
	$\tfrac{1}{2} X_{31}(a)  \rightarrow x_-(e_{11}(a)),$ & $\tfrac{1}{2} X_{32}(a)  \rightarrow x_-(e_{12}(a)),$  & $\tfrac{1}{2} X_{41}(a)  \rightarrow x_-(e_{12}(a)),$ & $ \tfrac{1}{2} X_{42}(a)  \rightarrow x_-(e_{22}(a))$,
	\end{tabular}\]where $\oplus e_{ij}(A)$ is the Peirce decomposition of $M_2(A)$, namely decomposing $2 \times 2$ matrices into subspaces corresponding to the four entries.
	
	Under these new expressions our defining presentation becomes \begin{align*}
	U \mapsto x_+(U) \mbox{ and } U \mapsto x_-(U) \mbox{ are } k \mbox{ linear maps,}\\
	[x_+(M_2(A)),x_+(M_2(A))]=[x_-(M_2(A)),x_-(M_2(A))]=0,\\
	[[x_+(U),x_-(V)],x_+(W)]=x_+(\tfrac{1}{2}(UVW+WVU)),\\
	[[x_-(V),x_+(U)],x_-(W)]=x_-(\tfrac{1}{2} (VUW+WUV)),
	\end{align*}for all $U,V,W\in M_2(A)$. Similar to Theorem \ref{sl2 non fp theorem}, this gives  $\widehat{\mathfrak{sl}_4(A)} \cong K(M_2(A)^{(+)})$. \end{proof}

\begin{theorem}
	Notations be as before. If $A$ is finitely presented as a $k$-algebra, then the special Jordan algebra $M_2(A)^{(+)}$ is finitely presented.
\end{theorem}

\begin{proof}
 Apply Proposition \ref{Tkk Jor presentation} to $ K(M_2(A)^{(+)})$ (which is isomorphic to $\widehat{\mathfrak{sl}_4(A)}$).
\end{proof}

	\section*{Acknowledgement}
The author would like to thank Efim Zelmanov for very useful discussions: his detailed comments shaped this note into its current form.

\bibliographystyle{plain}  
\bibliography{FPShenzhenmergedBibII}   

 \end{document}